\documentclass[12pt,a4paper]{article}

\usepackage{amsthm}
\usepackage{color}
\usepackage{amsmath}

\usepackage{amssymb}
\usepackage{epsfig}
\usepackage{pstricks}
\usepackage{pst-plot}
\usepackage{pst-node}
\usepackage{graphicx}
\usepackage[english]{babel}

\theoremstyle{plain}
\newtheorem{thm}{Theorem}
\newtheorem{prp}[thm]{Proposition}

\newtheorem{definition}{Definition}
\newtheorem{remark}[thm]{Remark}

\newcommand{\R}{\mathbb{R}}
\newcommand{\N}{\mathbb{N}}
\newcommand{\E}{\mathbb{E}}
\newcommand{\Z}{\mathbb{Z}}
\newcommand{\C}{\mathbb{C}}

\begin{document}
\title{On a Stochastic Leray-$\alpha$ model of Euler equations}

\author{David Barbato\footnote{Universit\`a di Padova, Dipartimento di Matematica, Padova, Italy,
barbato@math.unipd.it} \
\& Hakima Bessaih\footnote{University of Wyoming, Department of Mathematics, Dept. 3036, 1000
East University Avenue, Laramie WY 82071, United States, bessaih@uwyo.edu} \
 \& Benedetta Ferrario\footnote{Universit\`a di Pavia, Dipartimento di Matematica, via
  Ferrata 1, 27100 Pavia, Italy,
benedetta.ferrario@unipv.it}}
\maketitle

\begin{abstract}
We deal with the 3D inviscid Leray-$\alpha$ model.
The well posedness for this problem is not known; by adding
a random perturbation we prove that there exists a unique (in law)
global solution. 
The random forcing term  formally preserves  conservation of energy.
The result holds for initial velocity of finite energy 
and the solution has finite energy a.s.. 
These results are easily extended to the 2D case.
\end{abstract}

\section{Introduction}
The motion of incompressible fluids is described by the
Navier-Stokes equations
\begin{equation}\label{NSE}
\left\{
\begin{array}{l}
 \dfrac{\partial v}{\partial t} -\nu\Delta v+ (v \cdot \triangledown ) v +\triangledown p=f\\
\text{div}\ v=0
\end{array}
\right.
\end{equation}
for viscous fluids, or by the Euler equations
\begin{equation}\label{EE}
\left\{
\begin{array}{l}
 \dfrac{\partial v}{\partial t} + (v \cdot \triangledown ) v +\triangledown p=f\\
\text{div}\ v=0
\end{array}
\right.
\end{equation}
for inviscid fluids.

The unknown are the velocity field $v=v(t,x)$ and the pressure field
$p=p(t,x)$; $f$ is a given external force and $\nu>0$ is the viscosity
that corresponds to the inverse of the Reynolds number Re.
When the fluid moves in a bounded domain, suitable boundary conditions are associated to these equations, respectively, the no-slip and slip conditions.

The above two systems have a quite different behavior; for instance, 
when $f=0$ system \eqref{NSE} 
is  dissipative  while system \eqref{EE} is conservative.
It is well known, since the seminal work of Leray, that for initial
velocity of finite energy
the 3D Navier-Stokes system \eqref{NSE} has  a global weak solution
but its uniqueness is still an open problem. However, for the 3D Euler system
\eqref{EE} neither the global existence nor the uniqueness of global solutions are known
when the initial velocity is of finite energy (we refer to the review paper
  \cite{bardos-titi2007}  on this topic). 

We recall that to prove the existence of solutions to \eqref{NSE} 
in $\mathbb{R}^{d},\, d=2,3$, Leray \cite{leray}
considered the following regularization for $\alpha>0$
\begin{equation}\label{NSalpha}
\left\{
\begin{array}{l}
 \dfrac{\partial  v^{\alpha}}{\partial t}-\nu \Delta v^{\alpha}
 +(u^{\alpha} \cdot \triangledown ) v^{\alpha} +\triangledown p=f\\
u^{\alpha}=G_{\alpha} \ast v^{\alpha} \\ 
\text{div}\ v^{\alpha}=0
\end{array}
\right.
\end{equation}
where $G_{\alpha}$ is a smoothing kernel such that 
$u^{\alpha}\longrightarrow v^0$, in some sense,
as $\alpha\rightarrow 0$. In particular, system \eqref{NSalpha}
converges to the Navier-Stokes system \eqref{NSE} as
$\alpha\rightarrow 0$.
 
In \cite{Titi_leray}, a special smoothing kernel was considered, namely, the Green function associated  with the operator $1-\alpha\Delta$,
\begin{equation}\label{delta-alpha}
 u^{\alpha}=G_{\alpha} \ast v^{\alpha} =(1-\alpha\Delta)^{-1} v^{\alpha} 
 \end{equation}
for $\alpha>0$.
This kernel works as a kind of filter with width $\alpha$ and the parameter 
$\alpha$ reflects a sub-grid length scale in the model. 
This model was inspired by the Navier-Stokes $\alpha$ model (also known 
as the Camassa-Holm system or Lagrangian averaged Navier-Stokes
$\alpha$ equations) 
of turbulence, see \cite{titi1,titi2,titi3} and the references 
therein. Moreover,  it has been demonstrated analytically and computationally 
that the Navier-Stokes $\alpha$ model is a powerful tool in the study of turbulence, 
see \cite{titi3} and the reference therein. Along the same lines, it is worth 
mentioning that other $\alpha$ models, such as the Clark-$\alpha$ model 
\cite{titi_clarck} and  the Navier-Stokes Voigt equations \cite{titi_voigt}, 
have been used as a sub-grid scale models of turbulence.

In this paper, we are interested in a stochastic version of
 system \eqref{NSalpha} with $\nu=0$ and regularization given by 
\eqref{delta-alpha}, that is the following stochastic Leray-$\alpha$ model of
Euler equations
\begin{equation}\label{EEalpha}
\left\{
\begin{array}{l}
 d v^{\alpha}
 +[(u^{\alpha} \cdot \triangledown ) v^{\alpha} +\triangledown p]\ dt=
 ((\sigma\circ d W)\cdot \triangledown)  v^{\alpha}\\
v^{\alpha}=(1-\alpha \bigtriangleup) u^{\alpha} \\
\text{div}\ v^{\alpha}=0
\end{array}
\right.
\end{equation}
Here $W$ is a Brownian motion (in time) 
and $\circ d W$ referes to the  Stratonovitch 
differential; the parameter $\alpha$ is positive. When $\alpha=0$ the
first equation of \eqref{EEalpha} reduces to the stochastic Euler equation.

The well posedness of weak solutions of the deterministic 
system \eqref{EEalpha} ($\sigma=0$) is not known.
In particular, when the initial velocity has finite energy, 
existence of global weak  solutions 
can be proven  if $\alpha>0$, see the Appendix. 
However, the uniqueness is not known for both $d=2$ and $d=3$.
If $\alpha=0$ and $\sigma=0$, global existence of solutions are known 
for initial velocity of finite energy  and enstrophy for $d=2$ while it is 
open for $d=3$; their uniqueness is an open problem for $d=2$ and
$d=3$ (see, e.g., \cite{bardos-titi2007}).

Let us point out that 
the analysis of the deterministic Leray-$\alpha$ 
Euler equations in 3D, that is system (5) with $\sigma=0$, 
is more difficult than for the other approximations models for 
3D inviscid fluids (i.e., Camassa-Holm,  Clark and Voigt
models). Indeed,  for these other models there is formal conservation
of the sum $\int|v^\alpha|^2dx+\alpha \int |\nabla v^\alpha|^2 dx$ ($\alpha>0$), 
whereas the model we are interested in has only formal 
conservation of $\int|v^\alpha|^2dx$. From this point of view the 
deterministic Leray-$\alpha$ model  for 3D Euler equations 
considered in this paper is closer to the 3D Euler equations 
than the Voigt, Camassa-Holm and Clark models which regularize much
more the original Euler equations.

When adding an appropriate stochastic perturbation, we will prove that  
system \eqref{EEalpha} has a unique global solution (in law) when the initial  
condition is of finite energy. 
We will prove existence and uniqueness (in law) of solutions by means 
of Girsanov formula.
The multiplicative noise in \eqref{EEalpha} formally preserves 
the conservation of energy (see Section 3 for the details).
It is crucial to choose the random perturbation in \eqref{EEalpha} to
be written in the Stratonovitch form in such a way  that formally the energy of
the vector field $v^\alpha$ is conserved.

All our results are stated for a three dimensional spatial domain 
(a box $[0,2\pi]^3$, assuming periodic boundary conditions), but our proofs 
can be easily adapted to the two dimensional case. It would be interesting 
to study the behavior of the process $v^\alpha$ when $\alpha$ and/or
$\sigma$ converge to zero. This is  the subject of future research.

In the past few years, there has been a huge effort to tackle the problem 
of using a similar noise in order to improve the qualitative properties 
of some non linear equations.
In particular uniqueness of the stochastic equation has been proved either
when uniqueness is not known in the deterministic setting or with
weaker assumptions than in the deterministic setting; for these results,
see
\cite{attanasio-flandoli,BFM,BFM0,flandoli2009,flandoli2011,
flandoli-gubinelli-priola0,flandoli-gubinelli-priola1} 
and the references therein. We refer to 
\cite{CFM,flandoli-gubinelli-priola2} for the 2D Euler equations, 
and  to \cite{debussche-daprato, flandoli-romito} for some analysis on
the 3D Navier-Stokes equations.
For an overview of the problems and methods,
we refer to \cite{flandoli-stflour}.

As far as the content of this paper, in Section 2 we will introduce our 
functional setting and spaces. Section 3 will be devoted to the description 
of the stochastic Leray-$\alpha$ model in  Fourier components. 
We will write the model in both the Stratonovitch and It\^o forms. 
Section 4 will focus  on the linear model: global existence 
and uniqueness of strong 
(in the probabilistic sense) solutions will be proved.  The uniqueness proof 
is based on the study of a new linear problem constructed by means of the 
covariance matrix $A_{k}$.  The nonlinear model will be  studied in Section 5,
where we prove the existence and uniqueness of solutions  (in law) 
by means of a Girsanov formula.
In Section 6, our results  will be stated for the 
stochastic partial differential equation \eqref{EEalpha}.
To make  our paper  self-contained,
in the Appendix we will give the proof of global existence of weak solutions for
the deterministic Leray-$\alpha$ model of Euler equations.

\section{Functional setting}
Let the spatial domain be a torus $\mathbb T$, i.e. 
$x\in  \R^3$ and periodic boundary conditions on
the cube $[0,2\pi]^3$ are assumed. Notice that if $v$ is a solution,
then also $v+c$ ($c \in \mathbb R$)
is a solution. Therefore,
we consider mean zero velocity vectors, i.e.
$\int_{[0,2\pi]^3}v(t,x)\ dx=0$.

We fix notations. 
Given a complex number $z\in \C$, we denote, respectively,  
by $\Re z$ and $\Im z$ its
real and imaginary part; hence, $\overline z=\Re z-i \Im z$
and the product of two complex numbers $z$ and $w$ is 
$wz=(\Re w \Re z - \Im w \Im z)+ i (\Re w \Im z + \Im w \Re z)$.

Morever, let $x \in \C^3$ be represented as $x=(x^{(1)},x^{(2)},x^{(3)})$. 
For $x, y \in \C^3$, we set
$\langle x,y\rangle=\sum_{j=1}^{3}x^{(j)}\overline{y}^{(j)}$
and $\|x\|^2=\langle x,x\rangle$.
This defines, as a particular case, also the scalar product and the
norm in $\R^3$.

For each $k\in\Z^3$, let $e_k(x):=e^{i\langle x,k\rangle}$, $x \in \R^3$.
The family $\{e_k\}_{k\in\Z^3}$ is a complete orthogonal  basis
for the space $L^2 (\mathbb T,\C)$.

In this Section, we write the deterministic Leray-$\alpha$ model 
of the Euler system  
\eqref{EEalpha} 
\begin{equation}\label{EEdet}
\left\{
\begin{array}{l}
 \dfrac{\partial v^{\alpha}}{\partial t}
 +(u^{\alpha} \cdot \triangledown ) v^{\alpha} +\triangledown p=0\\
v^{\alpha}=(1-\alpha \bigtriangleup) u^{\alpha} \\
\text{div}\ v^{\alpha}=0,
\end{array}
\right.
\end{equation}
in Fourier components; this will be given in
\eqref{Eu-Fou}.
In the next Section, the stochastic forcing term will be introduced.

Assume $v(t,\cdot)$ and $u(t,\cdot)$ are in  $(L^2(\mathbb T,\C))^3$; then
$$
 v(t,x)=\sum_{k\in\Z^3} v_k(t) e_k(x)
 \qquad \text{ and  } \qquad
 u(t,x)=\sum_{k\in\Z^3} u_k(t) e_k(x)
$$
We have
$v_0=0$ and $u_0=0$, since $v$ has mean value zero. Moreover, since
$v$ and $u$ are real valued and $e_{-k}(x)=\overline{e_k(x)}$, we have
\[
  v_{-k}(t)=\overline{v_k}(t), \qquad u_{-k}(t)=\overline{u_k}(t).
\]
We set $\|v(t,\cdot)\|^2_{l^2}=\sum_k \|v_k(t)\|^2$.

We substitute in \eqref{EEdet}$_{2}$ and get
$$
 \sum_k v_k(t) e_k(x)=(1-\alpha \bigtriangleup) \sum_k u_k(t) e_k(x).
$$
From now on, we will drop the index $\alpha$ in the unknowns for simplicity.

Since $\bigtriangleup e_k(x)=-\|k\|^2 e_k(x)$ then
$ v_k(t)= \left( 1+\alpha \|k\|^2\right) u_k(t)$ and
$$
 \sum_k v_k(t) e_k(x)= \sum_k \left(1+\alpha \|k\|^2\right)u_k(t) e_k(x).
$$
From equation \eqref{EEdet}$_{3}$ we get the incompressibility condition
$$
 \langle v_k(t),k\rangle=0 \qquad \forall k\in \Z^3, \ t \in \mathbb R.
$$
Finally, using \eqref{EEdet}$_{1}$ we obtain
\[
\begin{split}
 \frac{d}{dt} &\sum_k v_k(t) e_k(x)
\\&=
 -\left(
 \sum_h u_h(t) e_h(x) \cdot  \triangledown \right) \sum_{k'} e_{k'}(x) v_{k'}(t)
 -\triangledown p(t,x)
\\&
 =-\sum_{h,k'} e_h(x) \left(
 u_h^{(1)}(t)\tfrac{\partial}{\partial x^{(1)}}+u_h^{(2)}(t)\tfrac{\partial}{\partial x^{(2)}}
 +u_h^{(3)}(t)\tfrac{\partial}{\partial x^{(3)}}\right) e_{k'}(x) v_{k'}(t)-\triangledown p(t,x)
\\&
 =-\sum_{h,k'} e_h(x) \left(
 u_h^{(1)}(t)i k'^{(1)}+u_h^{(2)}(t)i k'^{(2)}
 +u_h^{(3)}(t)i k'^{(3)}\right) e_{k'}(x) v_{k'}(t)-\triangledown p(t,x)
\\&
 =-i \sum_{h,k'} e_{h+k'}(x) \langle u_h(t),k'\rangle  v_{k'}(t)-\triangledown p(t,x),
\end{split}
\]
that is
$$
\frac{d}{dt} \sum_k v_k(t) e_k(x)=-i \sum_{h,k'} \frac{\langle v_h(t),k'\rangle}
{1+\alpha\|h\|^2}  P_k(v_{k'}(t)) e_{h+k'}(x)
$$
where
\begin{equation}\label{proiettore}
 P_k(v):= v - \frac{\left<v,k\right>}{\left<k,k\right>}\ k
\end{equation}
is the projection onto the space orthogonal to $k$.
\\
Summing up, since  $\langle v_h,k \rangle=\langle v_h,k-h\rangle$,
we obtain the system \eqref{EEdet}  written in Fourier  components
\begin{equation}\label{Eu-Fou}
\left\{
\begin{array}{l}
\dfrac{d v_k}{dt} (t) = -i \sum_{h \in \Z^3}
\dfrac{\langle v_h(t),k \rangle}{1+\alpha\|h\|^2}P_k(v_{k-h}(t))\\
\langle v_k(t),k \rangle=0\\
v_{-k}(t)=\overline{v_k}(t)
\end{array}
\right.
\end{equation}
for any $k$.

We notice that, given $\alpha \ge 0$, 
if $\sum_k \|v_k(t)\|^2<\infty$, then
the series in the r.h.s. of \eqref{Eu-Fou}$_{1}$ is convergent and
for any $k$ we have
\[
 \left\|\frac{dv_k}{dt}(t) \right\|\le \|v(t)\|_{l^2}^2  \|k\|.
\]

System \eqref{Eu-Fou} enjoys an important property: formally
the energy $E(t):=\frac 12 \sum_k \|v_k(t)\|^2$ is conserved under the
dynamics given by \eqref{Eu-Fou}. Indeed,
$$
\frac{d}{dt}  \|v_k(t)\|^2=\sum_{h}
2\Re
\left\{-i \frac{\left<v_h(t),k \right>}{1+\alpha\|h\|^2}
\langle P_k(v_{k-h}(t)),v_k(t)\rangle \right\}
$$
Summing over all components, we formally obtain
$$
 \frac{dE}{dt}  (t)=\sum_{k} \sum_{h}
 \Im \left\{ \frac{\langle v_h(t),k \rangle }{1+\alpha\|h\|^2}
 \langle v_{k-h}(t),v_k(t)\rangle\right\}
$$
which vanishes, since the sum contains terms which cancel each other
according to the following equality
\begin{equation}\label{eq:bilancio_formale}
\frac{\langle v_{h'},k' \rangle}{1+\alpha\|h'\|^2}
 \langle v_{k'-h'},v_{k'}\rangle
=\overline{\frac{\langle v_h,k \rangle}{1+\alpha\|h\|^2}
\langle v_{k-h},v_k\rangle}
\end{equation}
for $h'=-h$ and $k'=k-h$.

Let us finally notice that conservation of energy formally 
holds also for $\alpha=0$.

\section{The stochastic Leray-$\alpha$ model in Fourier components}
We are interested in a stochastic equation obtained from
\eqref{Eu-Fou} by adding a random forcing term in such a way that
energy is formally conserved. To this end we consider the system of
Stratonovich equations
\begin{equation}\label{eq:strat}
 dY_k(t) = -i \sum_{h\in\mathbb{Z}^3}P_k(Y_{k-h}(t))
 \frac{\langle Y_{h}(t)dt+\sigma \circ dW_h(t),k\rangle}{1+\alpha\|h\|^2} ,
 \qquad k \in \Z^3, k\neq \vec 0
\end{equation}
where
$\{W_h\}_{h\in \Z^3}$ is a family of independent $\C^3$-valued Brownian motions
on a filtered probability space $(\Omega,(\mathcal {F}_{t})_{t\ge 0}, P)$,
except for $\left<W_{h}(t),h\right>=0, W_{-h}(t)=\overline{W_h}(t)$.
According to the properties of Stratonovich integral 
(see \cite{Ikeda-Watanabe}) we have formally that $dE(t)=0 $. 
The computations are similar to the previous ones, 
using \eqref{eq:bilancio_formale} and 
\[
 \langle Y_{k-h},Y_k\rangle \langle \sigma \circ dW_{-h},k\rangle =
 \overline{\langle Y_{k},Y_{k-h}\rangle \langle \sigma \circ dW_{h},k\rangle}.
\]

Let us make precise the Stratonovich  formulation of system 
\eqref{eq:strat} in order to write it in terms of It\^o integrals.
Indeed, the
Stratonovitch formulation gives insights on the behaviour of the
system, but computations will be done on the It\^o formulation.

Set
$$
J:=\{(k_1,k_2,k_3)\in\Z^3 :
k_1>0 \text{ or } ( k_1=0, k_2>0) \text{ or } ( k_1=0, k_2=0, k_3>0)\}.
$$
Let $\{W'_h\}_{h\in J}$ be a family of independent $\C^3$-valued
standard Brownian motions. Define for any $h\in J$
\begin{equation}\label{da-w'-a-w}
 W_h:=P_h(W'_h),
\qquad  W_{-h}:=\overline{W}_{h}
\end{equation}
Therefore $\left<W_{h},h\right>=0, W_{-h}=\overline{W}_{h}$ for any $h \in \Z^3, h \neq\vec 0$.
Hence it is enough to give the family $\{W'_h\}_{h\in J}$ in order to define the stochastic part of system \eqref{eq:strat}.

Next, define
\begin{equation}\label{da-w-a-tildetilde}
\widetilde{\widetilde{W}}_{h,k}(t):=
\left<W_{h}(t),k\right>
\end{equation}
then each $\widetilde{\widetilde{W}}_{h,k}$ is a $\C$-valued Brownian
motion, whose real and imaginary part are independent.
Since
$\widetilde{\widetilde{W}}_{h,k}(t)=\left<W_{h}(t),k\right>
=\left<P_h(W'_h(t)),k\right>=\left<W'_h(t),P_h(k)\right>$,
then
$$
 Var[\Re\widetilde{\widetilde{W}}_{h,k}(1)]
 =Var[\Im\widetilde{\widetilde{W}}_{h,k}(1)]
 =\|P_{h}(k)\|^2=\|k\|^2 \sin^2(\theta)
$$
where  $\theta$ is the angle between $h$ and $k$.
Now, setting
\begin{equation}\label{da-tildetilde-a-tilde}
\widetilde{W}_{h,k} \| P_h(k)\| :=
\widetilde{\widetilde{W}}_{h,k}
\end{equation}
we have defined standard real Brownian motions $\Re\widetilde{W}_{h,k}$
and $\Im\widetilde{W}_{h,k}$.

Set
\begin{equation}\label{sigmah}
\sigma_h:=\dfrac{\sigma}{1+\alpha\|h\|^2}.
\end{equation}
Therefore \eqref{eq:strat} is
\begin{equation}
\begin{split}
 dY_k(t) =& -i\sum_{h\in\mathbb{Z}^3}
  \frac{\left<Y_{h}(t) ,k\right>}{1+\alpha\|h\|^2} P_k(Y_{k-h}(t))dt\\
 &  -i\sum_{h\in\mathbb{Z}^3}
  \sigma_{h} \|P_h(k)\| P_k(Y_{k-h}(t))\circ d\widetilde{W}_{h,k}(t)
 \label{eq:termine_di_stratonovich}
\end{split}
\end{equation}
Indeed, for the Stratonovich integral we have
\begin{equation}
\begin{split}
\int_0^t& P_k(Y_{k-h}(s))\left<\sigma_{h}\circ dW_{h}(s),k\right>
\\&=
\sigma_{h}P_k \int_0^t Y_{k-h}(s)\circ d\widetilde{\widetilde{W}}_{h,k}(s)\\
&= \sigma_{h} \|P_h(k)\| P_k \int_0^t Y_{k-h}(s)\circ
d\widetilde{W}_{h,k}(s)\\
&= \sigma_{h} \|P_h(k)\| P_k \int_0^t \Re Y_{k-h}(s)\circ d\Re\widetilde{W}_{h,k}(s)
  -\sigma_{h} \|P_h(k)\| P_k \int_0^t \Im Y_{k-h}(s)\circ
  d\Im\widetilde{W}_{h,k}(s)\\
& \;+i\sigma_{h}\|P_h(k)\| P_k \int_0^t \Re Y_{k-h}(s)
  \circ d\Im\widetilde{W}_{h,k}(s)
+i\sigma_{h} \|P_h(k)\| P_k \int_0^t \Im Y_{k-h}(s)\circ
 d\Re\widetilde{W}_{h,k}(s).
\end{split}
\end{equation}

We have the corresponding It\^o formulation.
\begin{thm}
Let $\{Y_k\}_{k \in \Z^3, k \neq \vec 0}$ be a sequence of  continuous and adapted 
processes defined on
a given filtered probability space such that $\sum_k \|Y_k(t)\|^2<\infty$ 
a.s.. 
If the sequence  solves the following system 
\begin{equation}\label{ito}
\begin{split}
 dY_k(t) = &-i\sum_{h\in\mathbb{Z}^3}
\frac{\langle Y_{h}(t) ,k\rangle}{1+\alpha\|h\|^2} P_k(Y_{k-h}(t))dt\\
&  -i\sum_{h\in\mathbb{Z}^3}  
\sigma_{h} \|P_h(k)\| P_k(Y_{k-h}(t)) d\widetilde{W}_{h,k}(t)\\
&  -\sum_{h\in\mathbb{Z}^3 }
\sigma_{h}^2 \|P_h(k)\|^2  P_k(P_{k-h}(Y_{k}(t)))dt, \;\qquad \forall k
\end{split}
\end{equation}
then it solves system
\eqref{eq:termine_di_stratonovich}.
\end{thm}
\proof
We are going to prove that when the It\^o integral in the r.h.s.
 is written as a Stratonovich integral, we get \eqref{eq:termine_di_stratonovich}.
The corrective term appearing in this transformation comes from
the quadratic variation
$[\sigma_{h} \|P_h(k)\| P_k(Y_{k-h}), \widetilde{W}_{h,k}]$ 
(see \cite{Ikeda-Watanabe}).

Let us work on the real and imaginary part; this makes the proof long
but clear.
Let $k'=k-h$; from \eqref{eq:termine_di_stratonovich} we have
\[
\begin{split}
\sigma_{h}\|P_h(k)\|&  P_k(\Re Y_{k'}(t))
= \sigma_{h}\|P_h(k)\|  P_k(\Re Y_{k'}(0))
+ \int_0^t \ldots \ldots ds\\
& +\sum_{h'\in\mathbb{Z}^3}
\int_0^t \sigma_{h}\|P_h(k)\|
P_k\left(\sigma_{h'}\|P_{h'}(k')\| P_{k'}(\Re Y_{k'-h'}(s))\right)
\circ d\Im\widetilde{W}_{h',k'}(s)
\\
& +\sum_{h'\in\mathbb{Z}^3}
\int_0^t \sigma_{h}\|P_h(k)\|
P_k\left(\sigma_{h'}\|P_{h'}(k')\| P_{k'}(\Im Y_{k'-h'}(s)) \right)
\circ d\Re\widetilde{W}_{h',k'}(s) .
\end{split}
\]
\[
\begin{split}
\sigma_{h}\|P_h(k)\|&  P_k(\Im Y_{k'}(t))
= \sigma_{h}\|P_h(k)\|  P_k(\Im Y_{k'}(0))
+ \int_0^t \ldots \ldots ds\\
& -\sum_{h'\in\mathbb{Z}^3} \int_0^t \sigma_{h}\|P_h(k)\|
P_k\left(\sigma_{h'} \|P_{h'}(k')\| P_{k'}(\Re Y_{k'-h'}(s))\right)
\circ d\Re\widetilde{W}_{h',k'}(s)
\\
& +\sum_{h'\in\mathbb{Z}^3} \int_0^t \sigma_{h}\|P_h(k)\|
P_k\left(\sigma_{h'} \|P_{h'}(k')\| P_{k'}(\Im Y_{k'-h'}(s))\right)
\circ d\Im\widetilde{W}_{h',k'}(s) .
\end{split}
\]
Bearing in mind that
$$
\widetilde{W}_{h',k-h} \text{\quad is independent of \quad }
\widetilde{W}_{h,k} \text{\quad \qquad if \ } h'\neq -h
\text{\ and \ } h'\neq h\
$$
and
$$
\widetilde{W}_{h',k-h} =
\overline{\widetilde{W}_{h,k}} \text{ \qquad if \ } h'= -h .
$$
$$
\widetilde{W}_{h',k-h} =
\widetilde{W}_{h,k} \text{ \qquad if \ } h'= h
$$
we are reduced to take into account only the terms with
$h'=-h$ and $h'=h$. Moreover, we use that 
$P_h(k)=P_{-h}(k-h)=P_h(k-h)$. Therefore
the corrective term for $\sigma_{h}\|P_h(k)\| P_k(\Re Y_{k-h}(s))
\circ d\Im\widetilde{W}_{h,k}(s)$ is
$$
-\frac{1}{2}\sigma_h^2\|P_h(k)\|^2
P_k\left(P_{k-h}\left(\Re Y_k(s) \right)\right)
+\frac{1}{2}\sigma_h^2\|P_h(k)\|^2
P_k\left(P_{k-h}\left(\Re Y_{k-2h}(s) \right)\right)\ ;
$$
that for $\sigma_{h}\|P_h(k)\| P_k(\Im Y_{k-h}(s))
\circ d\Re\widetilde{W}_{h,k}(s)$ is
$$
-\frac{1}{2}\sigma_h^2\|P_h(k)\|^2
P_k\left(P_{k-h}\left(\Re Y_k(s) \right)\right)
-\frac{1}{2}\sigma_h^2\|P_h(k)\|^2
P_k\left(P_{k-h}\left(\Re Y_{k-2h}(s) \right)\right)\ ;
$$
that for
$-\sigma_{h}\|P_h(k)\| P_k(\Re Y_{k-h}(s))
\circ d\Re\widetilde{W}_{h,k}(s)$ is
$$
-\frac{1}{2}\sigma_h^2\|P_h(k)\|^2
P_k\left(P_{k-h}\left(\Im Y_k(s) \right)\right)
-\frac{1}{2}\sigma_h^2\|P_h(k)\|^2
P_k\left(P_{k-h}\left(\Im Y_{k-2h}(s) \right)\right)\ ;
$$
that for
$\sigma_{h} \|P_h(k)\| P_k(\Im Y_{k-h}(s))
\circ d\Im\widetilde{W}_{h,k}(s)$ is
$$
-\frac{1}{2}\sigma_h^2\|P_h(k)\|^2
P_k\left(P_{k-h}\left(\Im Y_k(s) \right)\right)
+\frac{1}{2}\sigma_h^2\|P_h(k)\|^2
P_k\left(P_{k-h}\left(\Im Y_{k-2h}(s) \right)\right)\ .
$$
Summing up all the contributions, we get the expression given in the
Proposition. \hfill$\Box$

The aim of this paper is to study
the stochastic system \eqref{ito}
with initial data of finite energy.

\section{The linear model}
Let us consider
the linear system obtained by neglecting the nonlinear terms in \eqref{ito}:
\begin{equation}\label{eq:linear:_system}
 \left\{
 \begin{array}{llll}
  dY_k(t) &=&
  -i\sum_{h\in\mathbb{Z}^3}
 \sigma_{h} \|P_h(k)\| P_k(Y_{k-h}(t)) d\widetilde{B}_{h,k}(t)\\
 & &-\sum_{h\in\mathbb{Z}^3}
 \sigma_{h}^2 \|P_h(k)\|^2  P_k(P_{k-h}(Y_{k}(t)))dt \  & \\\ \\
 \left<Y_k(t),k\right> & = & 0 & \\
 &&&  \\
 Y_{-k}(t) & = & \overline{Y_k(t)}  \\\ \\
 Y_k(0) & = & y_k
 \end{array}
 \right.
\end{equation}
for each $k\neq \overrightarrow{0}$.
Here, $\{\widetilde{B}_{h,k}\}$ is a family of $\C$-valued Brownian motions obtained from
a family of independent $\C^3$-valued standard Brownian motions $\{B'_h\}_{h \in J}$
defined on a filtered probability space
$(\Omega, \{\mathcal F_t\}_{t \ge 0}, Q)$ with  the same procedure presented in
\eqref{da-w'-a-w}-\eqref{da-tildetilde-a-tilde}.

In the next section, we will see how Girsanov transform allows to pass
from the linear to the nonlinear system.

Notice that if $\E^Q\int_0^T \|Y(t)\|_{l^2}dt<\infty$, 
then the terms
in the r.h.s. of \eqref{eq:linear:_system}$_{1}$ are well defined. Indeed,
for the It\^o integrals we use that
\[\begin{split}
 \E^Q  \|\sum_{h\in\mathbb{Z}^3} \sigma_{h} \|P_h(k)\|
& \int_0^t P_k(Y_{k-h}(s) d\widetilde{B}_{h,k}(s) \|^2
\\
&=
 \E^Q  \sum_{h\in\mathbb{Z}^3}
 \sigma_{h}^2 \|P_h(k)\|^2 \int_0^t \|P_k(Y_{k-h}(s))\|^2 ds
\\&\le
 \sigma^2 \|k\|^2 \E^Q \int_0^t \sum_{h\in\mathbb{Z}^3}\|Y_{k-h}(s)\|^2 ds
\\&
 =\sigma^2 \|k\|^2 \E^Q \int_0^t\|Y(s)\|_{l^2}^2 ds
\end{split}\]
and for the deterministic integrals ($Q$-a.s.)
\[\begin{split}
 \| \int_0^t\sum_{h\in\mathbb{Z}^3}
 \sigma_{h}^2 \|P_h(k)\|^2 & P_k(P_{k-h}(Y_{k}(s)))\ ds \|
\\&\le \int_0^t\|\sum_{h\in\mathbb{Z}^3}
\sigma_{h}^2 \|P_h(k)\|^2  P_k(P_{k-h}(Y_{k}(s)))\|ds
\\
 &\le
 (\sum_{h\in\mathbb{Z}^3}\sigma_{h}^2)\|k\|^2 \int_0^t \|Y_{k}(s)\| ds
\end{split}\]

with $\sum_{h\in\Z^3} \sigma_{h}^2=
\sum_{h\in\Z^3} \frac{\sigma^2 }{\left(1+\alpha\|h\|^2\right)^2}<+\infty$.

We are interested in the stochastic system \eqref{eq:linear:_system} 
with deterministic initial data $y=\{y_k\}_k$ of finite energy.
We shall deal with strong solution in the probabilistic sense, that is the
filtered probability space $(\Omega, \{\mathcal F_t\}_{t \ge 0}, Q)$
and the Brownian motions $\{B'_h\}_{h \in J}$ are given a priori.
We shall prove existence and uniqueness of solutions of the following type.

\begin{definition}
Given $y \in l^2$, an energy controlled strong solution for system 
\eqref{eq:linear:_system}
is a family of continuous and adapted $\C^3$-valued stochastic processes 
$\{Y_k\}_{k\in\Z^3}$ such that
 for all $t \ge 0$
$$
 \left\{
 \begin{array}{llll}
  Y_k(t) &=&y_k
  -i\sum_{h\in\mathbb{Z}^3}
 \sigma_{h} \|P_h(k)\| \int_0^t P_k(Y_{k-h}(s)) d\widetilde{B}_{h,k}(s)\\
 & &-\sum_{h\in\mathbb{Z}^3}
 \sigma_{h}^2 \|P_h(k)\|^2 \int_0^t P_k(P_{k-h}(Y_{k}(s)))ds \  & \\\ \\
 \left<Y_k(t),k\right> & = & 0 & \\
 &&&  \\
 Y_{-k}(t) & = & \overline{Y_k(t)}  \\
 \end{array}
 \right.
$$
$Q$-a.s. for all $k$,
and for any $t >0$
$$
 \sum_{k\in\Z^3}\|Y_k(t)\|^2\leq \sum_{k\in\Z^3}\|y_k\|^2 \qquad Q-a.s.
$$
\end{definition}

\subsection{Existence of a strong solution}

\begin{thm}\label{teo-es-lin}
For any initial data  of finite energy, there exists an
energy controlled strong solution to system
 \eqref{eq:linear:_system}.
\end{thm}
\proof
We consider the finite dimensional system associated to
\eqref{eq:linear:_system}; for each integer $N>0$ this is obtained
by neglecting the components of index higher than $N$ in such a way
that the energy is conserved.
Set $\Gamma_N^k=\{h \in \Z^3: 0<\|h\|<N, 0<\|h-k\|<N\}$. Then the linear Galerkin system is
\begin{equation}\label{eq:linear:_system_finit}
\left\{
\begin{array}{llll}
 dY_k(t) &=&
 -i\sum_{h\in\Gamma_N^k}
 \sigma_{h} \|P_h(k)\| P_k(Y_{k-h}(t)) d\widetilde{B}_{h,k}(t)\\
 & &-\sum_{h\in\Gamma_N^k}
 \sigma_{h}^2 \|P_h(k)\|^2  P_k(P_{k-h}(Y_{k}(t)))dt \  & \\\ \\
 \langle Y_k(t),k\rangle & = & 0 & \\
&&&  \\
Y_{-k}(t) & = & \overline{Y(t)_k}  \\\ \\
Y_k(0) & = & y_k
\end{array}
\right.
\end{equation}
for each $k\in\Z^3, 0<\|k\|<N$.
We consider inital data $Y^N(0)=y^N$
obtained from the inital data $y$ of the full system \eqref{eq:linear:_system}
by putting to 0 the components $y_k$ with $\|k\|\ge N$.

By linearity the Galerkin system  has a unique global strong solution
$Y^N=\{Y^N_k\}_{0<\|k\|<N}$. Each component is a continuous and adapted process.
Moreover, energy is conserved, that is for any $t >0$
$$
 d \sum_{\|k\|<N} \|Y^N_k(t)\|^2=0 \qquad Q-\text{a.s.}
$$
To prove it, again we use the properties of the family 
$\{\widetilde{B}_{h,k}\}$ of Brownian motions and the projector
defined in \eqref{proiettore}.

Therefore, for an inital data of finite energy we have for any $t>0$
\begin{equation}\label{cons-energia}
 \|Y^N(t)\|_{l^2}=\|y^N\|_{l^2}\le \|y\|_{l^2} \quad Q-a.s.
\end{equation}
This implies that for any $p \in [1,\infty)$ we have
\begin{equation}\label{ineq-en}
  \mathbb E^Q \int_0^T \|Y^N(t)\|^p_{l^2}dt=\|y^N\|^p_{l^2}\le \|y\|^p_{l^2} \qquad
  \forall N .
\end{equation}
Therefore the sequence $\{Y^N\}_N$ is a bounded sequence
 in $L^p(\Omega\times[0,T];l^2)$
for any $1\le p\le\infty$.
This implies that there exists a sequence  $\{Y^{N_i}\}_{i=1}^\infty$ and a
process $Y\in L^\infty(\Omega\times[0,T];l^2)$ such that
\[
 \lim_{i\to \infty} Y^{N_i} =Y \;\text{ weakly in }L^p(\Omega\times[0,T];l^2) \qquad \text{ for } p<\infty
\]
and
\[
 \lim_{i\to \infty} Y^{N_i} =Y \;\; \star\text{-weakly in
 }L^\infty(\Omega\times[0,T];l^2) .
\]
In particular
\[
 \lim_{i\to \infty} Y_k^{N_i} =Y_k \text{ weakly in } L^2(\Omega\times[0,T]).
\]

Now we consider the convergence of the integrals in the r.h.s. 
of \eqref{eq:linear:_system_finit}$_1$.
\\
The It\^o integral, considered as  a linear operator,
 is strongly continuous in $L^2(\Omega\times [0,T])$; hence it is
 weakly continuous (see .e.g. \cite{krylov-rozovskii,pardoux}). 
This implies that each stochastic 
integral  converges weakly:
\[
 \lim_{N\to \infty} \int_0^\cdot P_k(Y^N_{k-h}(s))d\tilde B_{h,k}(s) 
 =\int_0^\cdot P_k(Y_{k-h}(s))d\tilde B_{h,k}(s)
\qquad \text{ weakly in }L^2(\Omega\times[0,T])
\]

On the other side, using the independence of the It\^{o} integrals 
and the It\^{o} isometry we have
\[\begin{split}
 \mathbb E^Q \| \sum_{\|h\|<N} \sigma_h \|P_h(k)\|&
 \int_0^T P_k(Y^{N}_{k-h}(s))  d\widetilde{B}_{h,k}(s)\|^2
\\&
 =\sum_{\|h\|<N} \sigma_h^2 \|P_h(k)\|^2 
   \mathbb E^Q\|\int_0^T P_k(Y^{N}_{k-h}(s)) d\widetilde{B}_{h,k}(s)\|^2 
\\&
 = \sum_{\|h\|<N} \sigma_h^2 \|P_h(k)\|^2 
   \mathbb E^Q \int_0^T \|P_k(Y^{N}_{k-h}(s))\|^2 ds
\\&
 \le \sigma^2 \|k\|^2 \sum_{\|h\|<N}  \mathbb E^Q \int_0^T \|Y^N_{k-h}(s)\|^2 ds
\\&
 \le \sigma^2 \|k\|^2 \|y\|^2_{l^2}\quad \text{ by } \eqref{ineq-en}.
\end{split}
\]
Hence
\[
 \lim_{N\to \infty}\sum_{\|h\|<N}  
 \int_0^\cdot \sigma_h \|P_h(k)\| P_k(Y^N_{k-h}(s))d\tilde B_{h,k}(s) 
 =
 \sum_{h} \int_0^\cdot \sigma_h \|P_h(k)\| P_k(Y_{k-h}(s))d\tilde B_{h,k}(s)
\]
weakly in $L^2(\Omega\times[0,T])$.

For the deterministic integral, it is an easy computation to identify the limit:
\[
\int_0^\cdot P_k(P_{k-h}(Y^N_k(s))) ds \to 
\int_0^\cdot P_k(P_{k-h}(Y_k(s))) ds \qquad \text{weakly in }
L^2(\Omega\times [0,T]) .
\]

For the limit $Y$, we have for any $t\ge 0$
\[
  \|Y(t)\|_{l^2} \le \|y\|_{l^2} \qquad Q-a.s.
\]
Therefore $Y$ is an energy controlled strong solution. 

Moreover, using again the estimate 
\begin{equation*}
  \mathbb E^{Q}\int_0^T \|Y(t)\|^2_{l^2}dt<\infty 
\end{equation*}
in a classical way
we obtain that the process given by the 
infinite sum of It\^o integrals $ \sum_{h} 
\sigma_h \|P_h(k)\| \int_0^tP_k(Y_{k-h}(s))d\tilde B_{h,k}(s)$ has a
continuous modification. 
Hence, we conclude that the process $Y$ has a
continuous modification.
\hfill $\Box$

\subsection{The covariance matrices}
In Theorem \ref{teo-es-lin} we have proved existence of energy controlled
solutions $Y=\{Y_k\}_{k \in \Z^3}$ of \eqref{eq:linear:_system}; 
now we want to show their
uniqueness. The idea is to study the time evolution of the covariance
matrices $\{A_k\}_{k \in \Z^3}$, defined as follows:
$$
A_k^{j_1,j_2}(t)=\E^{Q}\left[\Re Y_k^{(j_1)}(t) \Re Y_k^{(j_2)}(t)
+ \Im Y_k^{(j_1)}(t) \Im Y_k^{(j_2)}(t)\right]\qquad
j_1,j_2=1,2,3
$$
We collect the properties of $A_k$.
Since $\langle Y_k(t),k\rangle=0$ for any $t$ and $k$, then $k$ is an
eigenvector for $A_k(t)$ corresponding to the 0 eigenvalue.
$A_k(t)$ is a symmetric and semi-positive definite matrix; therefore the trace of $A_k(t)$ is non negative.
Moreover, we have
\begin{equation}
 \sum_{k\in\Z^3} \text{Tr}(A_k(t)) \le \|y\|_{l^2}^2
\end{equation}
for any $t \ge 0$.
Finally, $P_k A_k(t)P_k=A_k(t)$, where $P_k$ is the real matrix
previously defined in \eqref{proiettore}, which is symmetric semi-positive
definite; $P_k$ has the 0 eigenvalue  with eigenvector $k$ and the 
eigenvalue 1 of double multiplicity.

Bearing in mind \eqref{eq:linear:_system} and the properties of the Brownian motions $\tilde B_{h,k}$,
with some long but easy computations we get that
each $A_k$ fulfils a linear equation.
\begin{prp}\label{thm:A_k}
For each $k\neq \vec 0$,
$A_k$ fulfils the differential equation
\begin{equation}\label{eq:tensore}
\begin{split}
\frac{d A_k }{dt} (t)=
&-\sum_{h\in \Z^3} \sigma_h^2\|P_h(k)\|^2 P_k P_{k-h} A_k(t)\\
&-\sum_{h\in \Z^3} \sigma_h^2\|P_h(k)\|^2 A_k(t) P_{k-h} P_{k}\\
&+2\sum_{h\in \Z^3} \sigma_h^2\|P_h(k)\|^2 P_k A_{k-h}(t) P_{k}
\end{split}
\end{equation}
\end{prp}
This Proposition shows the non trivial fact that the covariance
matrices satisfy a closed differential system.

\subsection{Pathwise uniqueness}
We prove pathwise uniqueness for system \eqref{eq:linear:_system}.
\begin{thm}
 There exists at most one
energy controlled strong solution to system
 \eqref{eq:linear:_system}, that is given two
energy controlled strong solutions $Y_{[1]}$ and $Y_{[2]}$ to system
 \eqref{eq:linear:_system} defined on the same probability space
$(\Omega, \{\mathcal F_t\}_{t \ge 0}, Q)$ and
with the same initial data $y \in l^2$ and Brownian motions, we
 have for any $t\ge 0$
\[
 Y_{[1]}(t)=Y_{[2]}(t) \qquad Q-a.s.
\]
\end{thm}
\proof
Define
$$
 Y:=Y_{[1]}-Y_{[2]} .
$$
The idea of the proof is to take the difference $Y=Y_{[1]}-Y_{[2]}$;
by linearity $Y$ solves \eqref{eq:linear:_system} but with initial
data $Y(0)=0$. Let $\{A_k\}_{k \in \Z^3}$ be the covariance matrices  of $Y$; 
these matrices satisfy 
\eqref{eq:tensore} with zero initial condition and regularity
\eqref{tracciafinita}. Thus, the uniqueness problem for 
\eqref{eq:linear:_system} is
transformed in the easier uniqueness problem for the deterministic
system \eqref{eq:tensore}.
Indeed, in order to  show that for any $t>0$ we have $Y(t)=0$ $Q$-a.s.
it is enough to prove that system \eqref{eq:tensore}
with zero initial condition has the unique solution which vanishes,
i.e. for any $ k \neq \vec 0$,  given $A_k(0)=0$ we have
$A_k(t)=0$ for all $t>0$.

For any $T>0$ define
$$
B_k:=\int_0^T A_k(t) dt.
$$
Each tensor $B_k$ enjoys the same properties of $A_k$. Since
\begin{equation}\label{tracciafinita}
 \sum_{k\in\Z^3} \text{Tr}(A_k(t)) \le4 \|y\|_{l^2}^2 ,
\end{equation}
then
\begin{equation*}
\sum_{k\in \Z^3} \text{Tr}(B_k) \le 4 \|y\|_{l^2}^2 T;
\end{equation*}
thus
\begin{equation}\label{limiteBk}
 \lim_{k\to \infty} \text{Tr}(B_k)=0 .
\end{equation}

Writing equation \eqref{eq:tensore} in the integral form we have
\begin{equation}\label{eq:A_kT}
 A_k (T)=\sum_{h\in \Z^3} \sigma_h^2\|P_h(k)\|^2 \big[
 2 P_k  B_{k-h}  P_{k} - P_k  P_{k-h}  B_{k} -  B_{k} P_{k-h}  P_k \big] .
\end{equation}
If  $B_h=0$ for all $h$, then $A_k(T)=0$ and the proof is completed.

We proceed by contradiction. Suppose that there exists $k_1$ such that
$B_{k_1}$ does not vanish; then the maximal eigenvalue
$\lambda^*_{1}$ of $B_{k_1}$  is
positive. Starting from $B_{k_1}$ and $\lambda^*_{1}$ we construct a
sequence of $B_{k_n}$ and $\lambda^*_{n}$'s (with $\lambda^*_{n}$ the
maximal eigenvalue of $B_{k_n}$)
such that $\{\lambda^*_n\}_{n\in\N}$ is a stricly increasing sequence.
Therefore $\lim_{n \to \infty}\lambda^*_n>0$. On the other hand,
each $B_k$ is semipositive definite and therefore
$\text{Tr}(B_{k_n})\ge \lambda^*_n $. It follows that
\[
 \lim_{n\to \infty} \text{Tr}(B_{k_n}) >0
\]
which is impossible because of \eqref{limiteBk}.

Therefore we are left to prove that given $k_n \in \N$ such that
$B_{k_n}$ has maximal eigenvalue  $\lambda^*_{n}>0$, then there exists
$k_{n+1} \in \N$ such that
$B_{k_{n+1}}$ has maximal eigenvalue  $\lambda^*_{n+1}>\lambda^*_{n}>0$.

Let $\phi_n$ be the eigenvector corresponding to
$\lambda^*_{n}$. Therefore
\begin{equation}\label{autoval}
 B_{k_n} \phi_{n}= \lambda^*_{n} \phi_{n}
\end{equation}
Moreover $\phi_n$ is orthogonal to $k_n$, since they are eigenvectors
corresponding to different eigenvalues; therefore
\begin{equation}\label{ortog}
 P_{k_n}\phi_{n}= \phi_{n}
\end{equation}
From \eqref{eq:A_kT} we have
\begin{equation*}\begin{split}
0\le \langle \phi_{n}, A_{k_n}(T) \phi_{n}\rangle
=
\sum_{h\in \Z^3} \sigma_h^2\|P_h(k)\|^2 (&
2 \langle \phi_{n},P_{k_n} B_{k_n-h} P_{k_n}\phi_{n}\rangle
\\& -\langle  \phi_{n},P_{k_n} P_{k_n-h} B_{k_n}\phi_{n}\rangle
\\& -\langle  \phi_{n},B_{k_n} P_{k_n-h} P_{k_n}\phi_{n}\rangle )
\end{split}\end{equation*}
Using that each $B_k$ is symmetric and \eqref{autoval}-\eqref{ortog}, we get
$$
0\leq
 2 \sum_{h\in \Z^3} \sigma_h^2\|P_h(k)\|^2
      \big[\langle \phi_{n},B_{k_n-h}\phi_{n}\rangle
      -\lambda^*_n \langle \phi_{n}, P_{k_n-h}\phi_{n}\rangle\big]
$$
For fixed $n$, we have that $\langle \phi_{n},B_{k_n-h}\phi_{n}\rangle$ tends to 0
as $\|h\| \to \infty$,
because  of \eqref{limiteBk}. Therefore, some addends in the sum are
negative; if the sum must be non negative there must exist at least
one addend positive, i.e.
\[
 \exists \tilde h: \ \langle \phi_{n},B_{k_n-\tilde h}\phi_{n}\rangle
  -\lambda^*_n \langle \phi_{n}, P_{k_n-\tilde h}\phi_{n}\rangle >0
\]
Set $k_{n+1}=k_n-\tilde h$. Then
\[
 \langle \phi_{n},B_{k_{n+1}}\phi_{n}\rangle
 -\lambda^*_n \langle \phi_{n}, P_{k_{n+1}} \phi_{n}\rangle >0
\]
Setting $\psi_n = P_{k_{n+1}} \phi_{n}$ and using that
$B_{k_{n+1}}=P_{k_{n+1}}B_{k_{n+1}}P_{k_{n+1}} $
we get
\[
  \langle \psi_{n},B_{k_{n+1}} \psi_{n}\rangle >
   \lambda^*_n \langle \psi_{n}, \psi_{n}\rangle
\]
which implies that the maximal eigenvalue $\lambda^*_{n+1}$ of
$B_{k_{n+1}}$ is larger than $\lambda^*_n$.
\hfill $\Box$

\section{The nonlinear model}
Consider the nonlinear system in the It\^o form 
\begin{equation}\label{eq:nl}
 \left\{\begin{aligned}
 & dY_k(t) = -i\sum_{h\in\mathbb{Z}^3}
\frac{\langle Y_{h}(t) ,k\rangle}{1+\alpha\|h\|^2} P_k(Y_{k-h}(t))dt\\
 &  \qquad\qquad  -i\sum_{h\in\mathbb{Z}^3}
 \sigma_{h} \|P_h(k)\| P_k(Y_{k-h}(t)) d\widetilde{B}_{h,k}(t)\\
&  \qquad\qquad -\sum_{h\in\mathbb{Z}^3}
 \sigma_{h}^2 \|P_h(k)\|^2  P_k(P_{k-h}(Y_{k}(t)))dt \  \\
& \left<Y_k(t),k\right>  =  0  \\
& Y_{-k}(t) =  \overline{Y_k}(t)  \\
& Y_k(0)  = y_k
 \end{aligned}
 \right.
\end{equation}
Starting from the solution of the linear system \eqref{eq:linear:_system}
we construct a solution to this nonlinear system by means of Girsanov transform.
We shall deal with solutions on any fixed finite time interval $[0,T]$.

\begin{definition}\label{def_nonlinear}
Given $y\in l^2$, a weak solution of equation \eqref{eq:nl} in $l^{2}$
is a filtered probability space 
$(\Omega,\{\mathcal F_t\}_{t \in [0,T]}, P)$, 
 a sequence of independent $\C^3$-valued Brownian motions
$W'=\{W'_h\}_{h \in J}$ on $(\Omega,\{\mathcal F_t\}_{t \in [0,T]}, P)$
and an $l^{2}$-valued stochastic process $Y:=(Y_{k})_{k\in\mathbb Z^{3}}$ on
$(\Omega,\{\mathcal F_t\}_{t \in [0,T]}, P)$, 
with continuous and adapted components $Y_{k}$ such that
for all $t \in [0,T]$ 
$$
 \left\{
\begin{aligned}
  Y_k(t) =y_k -i&\sum_{h\in\mathbb{Z}^3}\int_0^t
\frac{\langle Y_{h}(s) ,k\rangle}{1+\alpha\|h\|^2} P_k(Y_{k-h}(s))ds
\\
&  -i\sum_{h\in\mathbb{Z}^3} \frac{\sigma}{1+\alpha\|h\|^2} 
  \|P_h(k)\| \int_0^t P_k(Y_{k-h}(s)) d\widetilde{W}_{h,k}(s)\\
&-\sum_{h\in\mathbb{Z}^3} \left( \frac{\sigma}{1+\alpha\|h\|^2}\right)^2
\|P_h(k)\|^2 \int_0^t P_k(P_{k-h}(Y_{k}(s)))ds  \\\ \\
 \left<Y_k(t),k\right>  =  0  \\
 Y_{-k}(t)  =  \overline{Y_k(t)}  \\
 \end{aligned}
 \right.
$$
$P$-a.s.
for each $k\in\mathbb{Z}^{3}$. We denote this solution by
$((\Omega,\{\mathcal F_t\}_{t \in [0,T]}, P),Y,W')$.

Moreover, it is called an energy controlled weak solution if for all $t\in[0,T]$
this solution satisfies 
$$
 \sum_{k\in\Z^3}\|Y_k(t)\|^2\le \sum_{k\in\Z^3}\|y_k\|^2 \qquad P-a.s.
$$
\end{definition}
As usual, the Brownian motions $\widetilde W_{h,k}$ 
are constructed by the family $\{W'_h\}_{h \in J}$ (see 
\eqref{da-w'-a-w}-\eqref{da-tildetilde-a-tilde}).

First we present the Girsanov result.
Let $Y=\{Y_h\}_h$ be the strong energy controlled solution of \eqref{eq:linear:_system}.
Define
\[
 L(t)=\frac{1}{\sigma} \sum_{h\in J}\int_{0}^{t} Y_{h}(s)d B_{h}(s)
\]
Then $L$ is a martingale and its quadratic
variation $[L,L]$ is well defined and given by
\begin{equation}\label{LLT}
 [L,L](t)=\frac 1{\sigma^2} \int_0^t \sum_{h\in J} Y_{h}(s)^2 ds
\end{equation}
We have
\begin{prp}
Let $Y$ be the strong solution of system \eqref{eq:linear:_system}
with the family of independent $\C^3$-valued
standard Brownian motions
$\{ B'_{h}\}_{h \in J}$ on $(\Omega,\{\mathcal F_t\}_{t\ge 0}, Q)$.
Then
\begin{equation}\label{gir}
 W'_h(t)=B'_h(t)-\frac 1\sigma \int_0^t Y_h(s) ds, \qquad h \in J,\ 0\le t\le T,
\end{equation}
defines a family of independent $\C^3$-valued
standard Brownian motions on $(\Omega,\{\mathcal F_t\}_{0\le t\le T}, P)$
with the measure $P$, defined on $(\Omega,\mathcal F_T)$, which is absolutely continuous with respect to the measure $Q$ and
\[
\frac{dP}{dQ}=e^{L(T)-\frac 12 [L,L](T)}
\]
\end{prp}
\proof
Since $\|Y(t)\|_{l^2}\le \|y\|_{l^2} \quad Q$-a.s., Novikov condition
\begin{equation}\label{eq:novikov}
 \mathbb{E}^Q
 \left[ e^{\frac 1{\sigma^2}\int_0^T\sum_{h\in J}\|Y_h(s)\|^2 ds }\right]<\infty
\end{equation}
holds true.
Therefore, Girsanov transform now applies in a classical way (see, e.g. 
\cite{Ikeda-Watanabe,Karatzas-Shreve}). 
\hfill $\Box$

We point out that $\frac{dP}{dQ}>0$, $Q$-a.s.; therefore also
$\frac{dQ}{dP}$ is well defined. 
Thus the measures $P$ and $Q$ 
are equivalent (i.e. absolutely continuous to each other).

Our main result is
\begin{thm}\label{thm_nonlinear}
For any initial data of finite energy, system \eqref{eq:nl}
 has an energy controlled  weak solution. Moreover, this solution is unique in law.
\end{thm}
\begin{proof}
As far as the existence is concerned, we have that \eqref{eq:linear:_system} 
has a unique strong solution
$Y$ defined on $(\Omega,\{\mathcal F_t\}, Q)$ and satisfying 
\eqref{eq:linear:_system}.
Using the Girsanov transform \eqref{gir}, we get that 
$((\Omega,\{\mathcal F_t\}, P),Y,W')$ is a weak solution
of \eqref{eq:nl}. Moreover, the measure $P$ is equivalent to the measure $Q$; then
$\|Y(t)\|_{l^2}\le \|y\|_{l^2}$ $P$- and $Q$-a.s. This means that this weak solution is an energy controlled solution.

As far as the uniqueness is concerned,
if there were two different weak solutions of the nonlinear system 
\eqref{eq:nl},
then each of them would give rise to a weak solution of the linear
system \eqref{eq:linear:_system};
these are obtained starting from $((\Omega,\{\mathcal F_t\}, P),Y,W')$ 
and getting $((\Omega,\{\mathcal F_t\}, Q),Y,B')$
via Girsanov transform.
On the other side, the pathwise uniqueness for the linear system 
\eqref{eq:linear:_system} implies the weak uniqueness; this comes from
Yamada-Watanabe theorem, which is usually known for finite
dimensional systems but whose validity holds also in the infinite dimensional
setting as soon as the It\^o stochastic integrals are well defined. Now, using
the absolute continuity of $P$ and $Q$, we deduce that the nonlinear 
system \eqref{eq:nl} has a unique solution (in law).
\end{proof}

\begin{remark}
i) The proof shows that our technique can be applied for any
$\alpha>0$  to more
general models, that is we can  deal
with a noise defined by means of 
$\sigma_h=\frac{\sigma }{\left(1+\alpha\|h\|^2\right)^{p}}$
and with  the smoothing term given by  
 $u^\alpha=(1-\alpha\Delta)^{-p}v^\alpha$, for any $p>4/3$.
\\ii) In the 2-dimensional case, all our computations can be extended
  for $ p>1/2$.
\end{remark}

\section{The formulation in SPDE}
The stochastic model considered so far in Fourier components
can be written as a stochastic partial differential equation, as follows
\begin{equation}\label{spde}
\left\{
\begin{array}{l}
  dv+(u \cdot \triangledown ) v\ dt\ +\triangledown p\ dt\ =
 \sum_{h\in\mathbb{Z}^3} \sum_{j=1}^3
 \sigma_h e_h \frac{\partial v}{\partial x^{(j)}} \circ dW_h^{(j)}
\\
 v=(1-\alpha \bigtriangleup) u \\
 \text{div}\ v=0\\
 v(0)=v_{0}
\end{array}
\right.
\end{equation}
For simplicilty we have dropped out the index $\alpha$ in the
unknowns.

The first equation of system \eqref{spde} can be
written in a more compact form as
\begin{equation}\label{SPDE}
 dv+\sum_{j=1}^3 \frac{\partial v}{\partial x^{(j)}}\ u^{(j)}\ dt\ 
 +\triangledown p\ dt\ =
 \sum_{j=1}^3
 \frac{\partial v}{\partial x^{(j)}} \circ dW^{(j)} .
\end{equation}
Here the random field $W$ is given as 
$W(t,x):=\sum_{h\in\mathbb{Z}^3} \sigma_h e_h(x) W_h(t)$.

More precisely, the Wiener process $W$ has the following form
\begin{equation}\label{noise}
 W(t,x)=2\sigma \sum_{h\in J}\frac{\cos (\langle h,x\rangle )\Re W_{h}(t)
 -\sin (\langle h,x\rangle ) \Im W_{h}(t)}{1+\alpha\|h\|^2},
\end{equation}
where $\{W_h\}_{h\in J}$ is a family of 
 independent $\C^3$-valued Brownian motions
on a filtered probability space $(\Omega,(\mathcal {F}_{t})_{t\ge 0}, P)$,
such that $\left<W_{h}(t),h\right>=0$.

Let us denote by $H$ and $V$ the subspaces of $(L^2(\mathbb T))^3$ and 
$(H^1(\mathbb T))^3$ respectively, given by 
vectors fields divergence free and  periodic: 
\[\begin{split}
 H&=\{ v\in \left(L^2(\mathbb T)\right)^{3},\quad  
      \nabla\cdot v=0, \quad \int_{\mathbb T}v(x)dx=0, 
      v \cdot n \text{ periodic on }\mathbb T\}
\\
V&=\{v \in H: v \in [H^1(\mathbb T)]^3,\quad u  \text{ periodic on }\mathbb T\}
\end{split}\]
where $n$ is the unit normal to the boundary of the spatial domain.

Moreover, identifying $H$ with its dual $H^\prime$ 
we get the Gelfand triple $(V,H,V^\prime)$
$$V\subset H\simeq H^\prime \subset V^\prime.$$
The norms are inheritaed from the spaces $(L^2(\mathbb T))^3$ and 
$(H^1(\mathbb T))^3$.

\begin{definition}\label{def_spde}
Given $v_{0}\in H$, a weak solution of \eqref{spde} in 
$H$
is a filtered probability space $(\Omega,\{\mathcal F_t\}, P)$,
 a sequence of independent $\C^3$-valued Brownian motions
$\{W_h\}_{h}$ on $(\Omega,\{\mathcal F_t\}, P)$
and an $H$-valued continuous and 
adapted stochastic process $v$ on
$(\Omega,\{\mathcal F_t\}, P)$,  such that
\begin{equation}\label{formadebole}
\begin{split}
 \int_{\mathbb{T}}\langle & v(t,x), \phi(x)\rangle dx 
 -\int_{0}^{t} \int_{\mathbb{T}} \langle (u(s,x)\cdot\nabla)\phi(x),
 v(s,x)\rangle  ds \ dx
\\&=
 \int_{\mathbb{T}}\langle v_{0}(x),\phi(x)\rangle dx\\
 & \qquad-\sum_{h\in\mathbb{Z}^3}    \sigma_h \sum_{j=1}^3 
 \int_{0}^{t}\Big(\int_{\mathbb{T}} e_h(x)
 \langle\frac{\partial \phi}{\partial x^{(j)}}(x),v(s,x)\rangle dx \Big)
 \circ dW_h^{(j)}(s), \; P-a.s.
\end{split}
 \end{equation}
for each $t\in [0,T]$ and for all test 
functions $\phi : \mathbb R^3\to \mathbb R^3$, 
periodic on $\mathbb T$, divergence free and of class $C^1$. 
We denote this solution by
$((\Omega,\{\mathcal F_t\}, P),v,W)$.\\
Moreover, it is called an energy controlled weak solution if for all $t \ge 0$
this solution satisfies
\[
 \|v(t,\cdot)\|_{H}\le \|v_0\|_{H}\qquad P-a.s.
\]
\end{definition}

This weak formulation corresponds to the stochastic equation
\eqref{SPDE}. Indeed, for a more regular solution $v$, by integration
by parts in \eqref{formadebole} one gets \eqref{SPDE}. 
This is a classical result for the Euler equation and the stochastic part uses
the properties of the Brownian motions.
Therefore we have the following result.
\begin{prp}
The equality $v(t,x)=\sum_{h\in\mathbb{Z}^3} Y_h(t) e_h(x)$
relates the solutions of the stochastic PDE \eqref{formadebole}
and  of the stochastic Fourier system
\eqref{eq:termine_di_stratonovich}
(or \eqref{ito}).
\end{prp}
Finally, we can reformulate our result for the SPDE:
\begin{thm}
For any initial velocity of finite energy, equation \eqref{SPDE}
 has an energy controlled  weak solution. Moreover, 
this solution is unique in law.
\end{thm}

\section{Appendix}

In this section, we present a  proof of
global existence of a weak solution for the deterministic system 
\eqref{EEdet} with initial velocity of finite energy;  
here weak has to be understood in the sense of PDE's. No written proof 
has been found in the published literature, whereas there are results
of local (in time) existence and uniqueness 
for very regular initial velocity; however, uniqueness of weak solutions is
an open problem. Anyway, Edriss Titi has 
presented a proof of global existence of weak solutions,
as a private communication. 

For simplicity, let us drop the index $\alpha$.

\begin{thm}\label{det}(Due to E. S. Titi)
Let  $v_{0}\in H$ and $T>0$. Then there exists a 
global weak  solution $v$ to the system \eqref{EEdet}
such that
$$v\in L^{\infty}([0,T]; H)\cap C([0, T]; V^\prime)$$
and 
\begin{equation}\label{weaksolution}
\begin{split}
 \int_{\mathbb{T}}\langle & v(t,x), \phi(x)\rangle dx 
 -\int_{0}^{t} \int_{\mathbb{T}} \langle (u(s,x)\cdot\nabla)\phi(x),
 v(s,x)\rangle  ds \ dx
\\&=
 \int_{\mathbb{T}}\langle v_{0}(x),\phi(x)\rangle dx\\
 \end{split}
 \end{equation}
for each $t\in [0,T]$ and for all test 
functions $\phi : \mathbb R^3\to \mathbb R^3$,
periodic with period box $\mathbb T$, divergence free and of class $C^1$. 
\end{thm}
\begin{proof}
Let $P_N$ be the finite dimensional projector that is $v^N=P_Nv$ means
$v^N(x)=\sum_{\|h\|<N}(\int_{\mathbb T} v(x) e_h(x) dx) e_h(x)$.
Then, we get the following finite-dimensional
 system system
\begin{equation}\label{sysm}
\begin{cases}
\dfrac{dv^{N}}{dt}+P_{N}[(u^{N}\cdot \nabla)v^{N}]=0&\\
u^{N}+\alpha \Delta u^{N}=v^{N}&
\end{cases}
\end{equation}
From $\eqref{sysm}_{1}$, we get  $\frac{d\;}{dt}\|v^N(t)\|^2_H=0$ for
any $N$ and $t$;
then $\|v^N(t)\|^2_H\le \|v_0\|^2_H$. Thus
\begin{equation}\label{u1}
 \sup_N\|v^N\|_{L^{\infty}(0,T;H)}<\infty;
\end{equation}
hence using $\eqref{sysm}_{2}$
\begin{equation}\label{v1}
 \sup_N \|u^N\|_{L^{\infty}(0,T;(H^{2}(\mathbb T))^3)}<\infty.
\end{equation}

Using again equation $\eqref{sysm}_{1}$
\begin{equation*}
\begin{split}
 \|\frac{dv^N(t)}{dt}\|_{V^\prime}&
  =\sup_{\|\phi\|_V=1}|\langle (u^N(t)\cdot \nabla)v^N(t),P_N\phi \rangle |\\
 &=\sup_{\|\phi\|_V=1}|\langle (u^N(t)\cdot \nabla)P_N\phi,v^N(t) \rangle |\\
 &\le \|u^N(t)\|_{(L^{\infty}(\mathbb T))^3}\|P_N\phi\|_V \|v^N(t)\|_{H}.
\end{split}
\end{equation*}
Using again \eqref{v1} and the embedding theorem 
$H^{2}(\mathbb T)\subset L^{\infty}(\mathbb T)$, 
we get that 
\begin{equation}\label{u2}
 \sup_N\|\frac{dv^N}{dt}(t)\|_{L^{\infty}(0,T;V^\prime)}<\infty.
\end{equation}
The estimate \eqref{u2} means that $\left\{v^N \right\}$ is 
uniformly Lipschitz  in $V^\prime$. 
On the other side using the estimate \eqref{u1}, 
$\left\{v^N(t)\right\}$ is inside a bounded ball of $H$.
Hence, the set $\left\{ v^N(t), \quad \forall t \right\}$ is a
compact subset of $V^\prime$. 
Using the Ascoli-Arzel\`{a} theorem, 
we can extract a subsequence called again $v^N(t)$ such that
\begin{equation*}
 v^N\longrightarrow v\quad {\rm in}\quad C([0,T]; V^\prime)
\end{equation*}
and $v\in  C([0,T]; V^\prime$. 
Moreover, using the estimate \eqref{u2}, the limit $v$ is $Lip([0,T]; V^\prime)$.

Using \eqref{u2} and \eqref{sysm}$_2$ we get that
\begin{equation}\label{v2}
 \sup_N \|\frac{du^N}{dt}(t)\|_{L^{\infty}(0,T;V)}<\infty;
\end{equation}
this result and \eqref{v1} allow to use
the compactness theorem (Aubin-Lions). Therefore 
we can extract a subsequence, again called $u^N$ such that
\begin{equation*}
 u^N\longrightarrow u=(1-\alpha\Delta)^{-1} v\quad {\rm in}
  \quad L^{p}([0,T]; H^{2-\epsilon}),
\end{equation*}
for some arbitrary $\epsilon>0$ and $p$ finite. We shall take
$\epsilon<\frac 12$ in order to use that $H^{2-\epsilon}(\mathbb T)
\subset L^\infty(\mathbb T)$.
Now, we have all the ingredients to pass to the limit in the system 
 \eqref{sysm} that we are going to write in the weak form:
\begin{equation*}
\int_{\mathbb T}  \langle v^N(t,x)-v^N(s,x),\phi(x) \rangle dx
 =\int_{s}^{t}\int_{\mathbb T}\langle (u^N(r,x)\cdot\nabla)P_N \phi(x),
 v^N(r)\rangle dx \ dr.
\end{equation*}
It is  easy  to pass to the limit on the l.h.s. 
of the above equality. Let us focus of the r.h.s. of above equality:
the non linear term 
$\langle (u^N \cdot\nabla) P_N \phi, v^N\rangle$
converges  weakly in $L^1(0,T;V^\prime)$, since
$u^N$ converges strongly in $L^2(0,T;[L^\infty(\mathbb T)]^3)$
and $v^N$ converges weakly in $L^2(0,T;H)$ (due to \eqref{u1},
considering possibly a new subsequence). 
\end{proof}

{\bf Acknowledgment:} We are very grateful to Professors  
Franco Flandoli and Edriss S. Titi for various stimulating conversations.
The work of H. Bessaih was supported in part by the
GNAMPA-INDAM project "Professori Visitatori". We would like to thank
the hospitality of the Department of Mathematics of the University of Pavia
where part of this research started and the IMA in Minneapolis where 
the paper has been finalized.

\end{document}